\documentclass[12pt]{article}
\usepackage{amssymb}
\usepackage{amsthm}
\usepackage{amsmath}
\usepackage{graphicx}
\usepackage{enumerate}

\newtheorem{theorem}{Theorem}

\newtheorem{lemma}[theorem]{Lemma}
\newtheorem{proposition}[theorem]{Proposition}
\newtheorem{remark}[theorem]{Remark}
\newtheorem{example}[theorem]{Example}

\title{Compositions and decompositions of binary relations}
\author{Ivan~Chajda and Helmut~L\"anger}
\date{}
\begin{document}

\footnotetext{Support of the research of the first two authors by the Austrian Science Fund (FWF), project I~4579-N, and the Czech Science Foundation (GA\v CR), project 20-09869L, entitled ``The many facets of orthomodularity'', as well as by \"OAD, project CZ~02/2019, entitled ``Function algebras and ordered structures related to logic and data fusion'', and, concerning the first author, by IGA, project P\v rF~2021~030, is gratefully acknowledged.}

\maketitle

\begin{abstract}
It is well-known that to every binary relation on a non-void set $I$ there can be assigned its incidence matrix, also in the case when $I$ is infinite. We show that a certain kind of ``multiplication'' of such incidence matrices corresponds to the composition of the corresponding relations. Using this fact we investigate the solvability of the equation $R\circ X=S$ for given binary relations $R$ and $S$ on $I$ and derive an algorithm for solving this equation by using the connections between the corresponding incidence matrices. Moreover, we describe how one can obtain the incidence matrix of a product of binary relations from the incidence matrices of its factors.
\end{abstract}

{\bf AMS Subject Classification:} 08A02, 08A05

{\bf Keywords:} binary relation, incidence matrix, composition of incidence matrices, decomposition of binary relations, solving of relational equations

A systematic study of binary relations is a rather old task initiated in papers by J.~Riguet (\cite R) and R.~Fraiss\'e, see e.g.\ \cite{F53} and \cite{F54}. An algebraic approach to binary relations was introduced and developed by B.~J\'onsson (\cite J). An approach via assigned groupoids was started by the authors in the relatively recent papers \cite{CL13} and \cite{CL16} and, together with P.~\v Sev\v c\'ik, in \cite{CLS}.

The aim of the present paper is to show how the incidence matrices of given binary relations are useful for constructing relational products and decomposing a given relation into a relational product of two relations where one factor is given. As a byproduct we describe the incidence matrix of the Cartesian product of a set of given binary relations.

In the following let $I$ be a set. Then the Kronecker delta $\delta_{ij}$ on $I$ is defined by
\[
\delta_{ij}:=\left\{
\begin{array}{ll}
1 & \text{if }i=j, \\
0 & \text{otherwise}
\end{array}
\right.
\]
for all $i,j\in I$. Let $L$ be a further set. By an $I\times I$-matrix $M=[a_{ij}]$ over $L$ we mean a mapping $(i,j)\mapsto a_{ij}$ from $I\times I$ to $L$. If $I$ is finite we assume $I=\{1,\ldots,n\}$ and call the matrix an $n\times n$-matrix over $L$. Let $L^{I\times I}$, respectively $L^{n\times n}$, denote the set of all $I\times I$-matrices, respectively $n\times n$-matrices, over $L$.

To every binary relation $R$ on $I$ we assign its {\em incidence matrix} $M_R=[a_{ij}]\in\{0,1\}^{I\times I}$ as follows:
\[
a_{ij}:=\left\{
\begin{array}{ll}
1 & \text{if }(i,j)\in R, \\
0 & \text{otherwise}.
\end{array}
\right.
\]
For $I\times I$-matrices $A=[a_{ij}]$ and $B=[b_{ij}]$ over $\{0,1\}$ let $A\odot B$ denote the $I\times I$-matrix $C=[c_{ij}]$ over $\{0,1\}$ defined by
\[
c_{ij}:=\max_{k\in I}a_{ik}b_{kj}
\]
for all $i,j\in I$, i.e.\ $A\odot B$ is analogously defined as the usual matrix product $A\cdot B$, only the addition operation is replaced by the maximum operation (which works also for infinite $I$). It means that we ``multiply'' the $i$-th row of $A$ with the $j$-th column of $B$ using this kind of ``addition''.

At first, we will study the composition of binary relations and their incidence matrices. It was already mentioned in \cite R that a certain composition of incidence matrices corresponds to the product of the corresponding relations. However, an explicit form of such a composition was not presented. We can state and prove the following elementary result.

\begin{proposition}\label{prop1}
Let $R,S\subseteq I\times I$. Then $M_{R\circ S}=M_R\odot M_S$.
\end{proposition}

\begin{proof}
Put $M_R=[a_{ij}]$, $M_S=[b_{ij}]$, $M_{R\circ S}=[c_{ij}]$ and $M_R\odot M_s=[d_{ij}]$ and let $k,l\in I$. Then the following are equivalent:
\begin{align*}
                                      c_{kl} & =1, \\
                                       (k,l) & \in R\circ S, \\
\text{there exists some }m\in I\text{ with } & (k,m)\in R\text{ and }(m,l)\in S, \\
\text{there exists some }m\in I\text{ with } & a_{km}=b_{ml}=1, \\
                   \max_{m\in I}a_{km}b_{ml} & =1, \\
                                      d_{kl} & =1.
\end{align*}
This shows $M_{R\circ S}=M_R\odot M_S$.
\end{proof}

For $I\times I$-matrices $A=[a_{ij}]$ and $B=[b_{ij}]$ over $\{0,1\}$ let $A\oplus B$ denote the $I\times I$-matrix $C=[c_{ij}]$ over $\{0,1\}$ defined by
\[
c_{ij}:=\max(a_{ij},b_{ij})
\]
for all $i,j\in I$. Moreover, let $M_0$ and $M_1$ denote the $I\times I$-matrices $[0]$ and $[\delta_{ij}]$ over $\{0,1\}$ and put $\Delta:=\{(x,x)\mid x\in L\}$.

With the knowledge how to compose incidence matrices at hand, we can describe an algebraic structure on the set of all incidence matrices of a given dimension. Let us note that the structure of the set of binary relations on a given set with respect to relational operations (product, union, complementation etc.) was originally described by B.~J\'onsson, see e.g.\ \cite J and references therein.

Recall that a {\em unitary semiring} is an algebra $(S,+,\cdot,0,1)$ of type $(2,2,0,0)$ satisfying the following conditions:
\begin{itemize}
\item $(S,+,0)$ is a commutative monoid,
\item $(S,\cdot,1)$ is a monoid,
\item $(x+y)z\approx xz+yz$ and $z(x+y)\approx zx+zy$,
\item $x0\approx0x\approx0$.
\end{itemize}

\begin{theorem}\label{th1}
Let $I$ be a set. Then
\begin{enumerate}[{\rm(i)}]
\item $(2^{I\times I},\cup,\circ,\emptyset,\Delta)$ is a unitary semiring,
\item the mapping $R\mapsto M_R$ from $2^{I\times I}$ to $\{0,1\}^{I\times I}$ is an isomorphism from $(2^{I\times I},\cup,\circ,\emptyset,$ $\Delta)$ to $(\{0,1\}^{I\times I},\oplus,\odot,M_0,M_1)$ and hence the latter algebra is a unitary semiring, too.
\end{enumerate}
\end{theorem}

\begin{proof}
Let $R,S\subseteq I\times I$, $M_R=[a_{ij}]$, $M_S=[b_{ij}]$, $M_{R\cup S}=[c_{ij}]$, $M_R\oplus M_S=[d_{ij}]$ and $k,l\in I$.
\begin{enumerate}[(i)]
\item Obviously, $(2^{I\times I},\cup,\emptyset)$ is a commutative monoid, $(2^{I\times I},\circ,\Delta)$ a monoid and $R\circ\emptyset=\emptyset\circ R=\emptyset$. The distributivity laws can be easily verified.
\item The following are equivalent:
\begin{align*}
             c_{kl} & =1, \\
              (k,l) & \in R\cup S, \\
         (k,l)\in R & \text{ or }(k,l)\in S, \\
           a_{kl}=1 & \text{ or }b_{kl}=1, \\
\max(a_{kl},b_{kl}) & =1, \\
             d_{kl} & =1
\end{align*}
showing $M_{R\cup S}=M_R\oplus M_S$. From Proposition~\ref{prop1} we know that $M_{R\circ S}=M_R\odot M_S$. Since, finally, $M_\emptyset=M_0$ and $M_\Delta=M_1$, the mapping $R\mapsto M_R$ from $2^{I\times I}$ to $\{0,1\}^{I\times I}$ is a homomorphism from $(2^{I\times I},\cup,\circ,\emptyset,\Delta)$ to $(\{0,1\}^{I\times I},\oplus,\odot,M_0,M_1)$. Since $f$ is bijective, it is an isomorphism. The last assertion is clear.
\end{enumerate}
\end{proof}

\begin{remark}
From Theorem~\ref{th1} we conclude that $\odot$ is associative and distributive with respect to $\oplus$.
\end{remark}

Now we turn our attention to the decomposition of binary relations. Consider two binary relations $R$ and $S$ on a given set $A$. We ask if there exists a binary relation $X$ on $A$ satisfying the equation
\[
R\circ X=S,
\]
i.e.\ we ask if $S$ can be decomposed into the given relation $R$ and a certain (unknown) relation $X$. At first we present an example showing that such a relation $X$ can be found by using the composition of incidence matrices presented above. It is a method similar to solving sets of linear equations over the two-element field but instead of the binary addition we use the binary operation $\max$ as explained before.

\begin{example}\label{ex1}
Put
\begin{align*}
n & :=3, \\
R & :=\{(1,2),(2,1),(3,2),(3,3)\}, \\
S & :=\{(1,1),(1,2),(3,1),(3,2)\}.
\end{align*}
We consider the equation $R\circ X=S$. This is equivalent to $M_R\odot M_X=M_S$ where
\[
M_R=\left(
\begin{array}{ccc}
0 & 1 & 0 \\
1 & 0 & 0 \\
0 & 1 & 1
\end{array}
\right), M_X=\left(
\begin{array}{ccc}
x_{11} & x_{12} & x_{13} \\
x_{21} & x_{22} & x_{23} \\
x_{31} & x_{32} & x_{33}
\end{array}
\right)\text{ and }M_S= \left(
\begin{array}{ccc}
1 & 1 & 0 \\
0 & 0 & 0 \\
1 & 1 & 0
\end{array}
\right).
\]
We obtain immediately
\[
(x_{21},x_{22},x_{23},x_{11},x_{12},x_{13})=(1,1,0,0,0,0)
\]
and, using our computation,
\begin{align*}
\max(x_{21},x_{31}) & =1, \\
\max(x_{22},x_{32}) & =1, \\
\max(x_{23},x_{33}) & =0.
\end{align*}
The last three equations are equivalent to
\begin{align*}
\max(1,x_{31}) & =1, \\
\max(1,x_{32}) & =1, \\
\max(0,x_{33}) & =0
\end{align*}
and hence to $x_{33}=0$. This shows that the equation $R\circ X=S$ has exactly four solutions, namely
\begin{align*}
X & =\{(2,1),(2,2)\}, \\
X & =\{(2,1),(2,2),(3,1)\}, \\
X & =\{(2,1),(2,2),(3,2)\}, \\
X & =\{(2,1),(2,2),(3,1),(3,2)\}.
\end{align*}
\end{example}

Next we show how the incidence matrix of the Cartesian product of binary relations over different base sets can be derived from the incidence matrices of the corresponding factors. For this we introduce the following kind of product of relations over different base sets.

Let $(I_k)_{k\in K}$ be a family of sets, put $I:=\prod\limits_{k\in K}I_k$ and for all $k\in K$ let $R_k\subseteq I_k\times I_k$. Define
\[
\prod_{k\in K}R_k:=\{\big((i_k)_{k\in K},(j_k)_{k\in K}\big)\in I\times I\mid(i_k,j_k)\in R_k\text{ for all }k\in K\}.
\]

\begin{theorem}
Let $(I_k)_{k\in K}$ be a family of sets, and for all $k\in K$ let $R_k\subseteq I_k\times I_k$ and $M_{R_k}=[a_{i_kj_k}]$. Put $I:=\prod\limits_{k\in K}I_k$ and $R:=\prod\limits_{k\in K}R_k$. Then $R\subseteq I\times I$. Let $M_R=[a_{ij}]$. Then
\[
a_{(i_k)_{k\in K}(j_k)_{k\in K}}=\min_{k\in K}a_{i_kj_k}
\]
for all $(i_k)_{k\in K},(j_k)_{k\in K}\in I$.
\end{theorem}

\begin{proof}
Let $(l_k)_{k\in K},(m_k)_{k\in K}\in I$. Then the following are equivalent:
\begin{align*}
a_{(l_k)_{k\in K}(m_k)_{k\in K}} & =1, \\
\big((l_k)_{k\in K},(m_k)_{k\in K}\big) & \in R, \\
(l_k,m_k) & \in R_k\text{ for all }k\in K, \\
a_{l_km_k} & =1\text{ for all }k\in K, \\
\min_{k\in K}a_{l_km_k} & =1.
\end{align*}
This shows
\[
a_{(l_k)_{k\in K}(m_k)_{k\in K}}=\min_{k\in K}a_{l_km_k}.
\]
\end{proof}

\begin{example}
If
\begin{align*}
I & :=\{1,2\}, \\
J & :=\{1,2,3\}, \\
R & :=\{(1,1),(2,1)\}\subseteq I\times I, \\
S & :=\{(3,2),(3,3)\}\subseteq J\times J
\end{align*}
then
\begin{align*}
M_R & =\left(
\begin{array}{cc}
1 & 0 \\
1 & 0
\end{array}
\right), \\
M_S & =\left(
\begin{array}{ccc}
0 & 0 & 0 \\
0 & 0 & 0 \\
0 & 1 & 1
\end{array}
\right), \\
  K & :=I\times J=\{(1,1),(1,2),(1,3),(2,1),(2,2),(2,3)\}, \\
  T & :=R\times S=\{\big((1,3),(1,2)\big),\big((1,3),(1,3)\big),\big((2,3),(1,2)\big),\big((2,3),(1,3)\big)\}\subseteq K\times K, \\
M_T & =\left(
\begin{array}{cccccc}
0 & 0 & 0 & 0 & 0 & 0 \\
0 & 0 & 0 & 0 & 0 & 0 \\
0 & 1 & 1 & 0 & 0 & 0 \\
0 & 0 & 0 & 0 & 0 & 0 \\
0 & 0 & 0 & 0 & 0 & 0 \\
0 & 1 & 1 & 0 & 0 & 0
\end{array}
\right).
\end{align*}
\end{example}

In the next theorem we present sufficient but not necessary conditions for solving the equation $R\circ X=S$. That these conditions are not necessary can be seen by the fact that the mapping $f$ mentioned in Theorem~\ref{th1} does not exist in Example~\ref{ex1} though the equation $R\circ X=S$ is solvable.

\begin{theorem}\label{th2}
Let $R,S\subseteq I\times I$, $f\colon I\rightarrow I$ and assume $M_R=[\delta_{j,f(i)}]$ and $M_S=[b_{ij}]$.
\begin{enumerate}[{\rm(i)}]
\item The equation $R\circ X=S$ has a solution if and only if
\[
\text{for all }j,k,l\in I\text{ we have }b_{kj}=b_{lj}\text{ whenever }f(k)=f(l).
\]
In this case $X$ with $M_X=[x_{ij}]$ is a solution if and only if
\[
\text{for all }i,j\in I\text{ we have }x_{f(i),j}=b_{ij}.
\]
\item If $f$ is bijective then the equation $R\circ X=S$ has exactly one solution, namely $X$ with $M_X=[b_{f^{-1}(i),j}]$.
\end{enumerate}
\end{theorem}

\begin{proof}
Let $X\subseteq I\times I$ and $M_X=[x_{ij}]$.
\begin{enumerate}[(i)]
\item Then the following are equivalent:
\begin{align*}
                          R\circ X & =S, \\
                      M_R\odot M_X & =M_S, \\
\max_{k\in I}\delta_{k,f(i)}x_{kj} & =b_{ij}\text{ for all }i,j\in I, \\
                        x_{f(i),j} & =b_{ij}\text{ for all }i,j\in I.
\end{align*}
\item If $f$ is bijective then the following are equivalent:
\begin{align*}
x_{f(i),j} & =b_{ij}\text{ for all }i,j\in I, \\
    x_{ij} & =b_{f^{-1}(i),j}\text{ for all }i,j\in I.
\end{align*}
\end{enumerate}
\end{proof}

How the mapping from Theorem~\ref{th2} works is illustrated in the following example.

\begin{example}
Put
\begin{align*}
I & :=\{1,2,3\}, \\
R & :=\{(1,2),(2,3),(3,1)\}, \\
S & :=\{(1,1),(1,2),(2,3),(3,3)\}.
\end{align*}
Then $M_R=\left(
\begin{array}{ccc}
0 & 1 & 0 \\
0 & 0 & 1 \\
1 & 0 & 0
\end{array}
\right)$ and $M_S=\left(
\begin{array}{ccc}
1 & 1 & 0 \\
0 & 0 & 1 \\
0 & 0 & 1
\end{array}
\right)$. There is only one possibility for the mapping f, namely
\begin{align*}
               \big(f(1),f(2),f(3)\big) & =(2,3,1), \\
\big(f^{-1}(1),f^{-1}(2),f^{-1}(3)\big) & =(3,1,2).
\end{align*}
Since $f$ is a bijection, the equation $R\circ X=S$ has the unique solution $X$ with $M_X=\left(
\begin{array}{ccc}
0 & 0 & 1 \\
1 & 1 & 0 \\
0 & 0 & 1
\end{array}
\right)$, i.e.\ $X=\{(1,3),(2,1),(2,2),(3,3)\}$.
\end{example}

We are going to show several cases in which the equation $R\circ X=S$ is not solvable.

\begin{lemma}
Let $R,S\subseteq I\times I$, $M_R=[a_{ij}]$, $M_S=[b_{ij}]$ and $k,l,m,n\in I$ and assume that one of the following conditions holds:
\begin{enumerate}[{\rm(i)}]
\item $a_{kj}=0$ for all $j\in I$ and there exists some $p\in I$ with $b_{kp}\neq0$,
\item $a_{kj}=a_{lj}$ for all $j\in I$ and there exists some $p\in I$ with $b_{kp}\neq b_{lp}$,
\item $a_{kj}=\delta_{jl}$ for all $j\in I$, $a_{nl}=b_{km}=1$ and $b_{nm}=0$.
\end{enumerate}
Then the equation $R\circ X=S$ has no solution.
\end{lemma}

\begin{proof}
Assume $X$ to be a solution with $M_X=[x_{ij}]$. Then $M_R\odot M_X=M_S$. Now we have
\begin{enumerate}[(i)]
\item $b_{kp}=\max\limits_{r\in I}a_{kr}x_{rp}=0$, a contradiction.
\item $b_{kp}=\max\limits_{r\in I}a_{kr}x_{rp}=\max\limits_{r\in I}a_{lr}x_{rp}=b_{lp}$, a contradiction.
\item $x_{lm}=\max\limits_{j\in I}\delta_{jl}x_{jm}=\max\limits_{j\in I}a_{kj}x_{jm}=b_{km}=1$ which implies
\[
1=a_{nl}x_{lm}\leq\max_{j\in I}a_{nj}x_{jm}=b_{nm}=0,
\]
a contradiction.
\end{enumerate}
\end{proof}

In the following proposition we present a case where the equation $R\circ X=S$ can be easily solved.

\begin{proposition}
Let $R,S\subseteq I\times I$ with $R\subseteq S$ and assume $R$ to be reflexive and $S$ to be transitive. Then the equation $R\circ X=S$ has a solution, namely $X=\Delta\cup(S\setminus R)$.
\end{proposition}

\begin{proof}
Let $(a,b)\in R\circ\big(\Delta\cup(S\setminus R)\big)$. Then there exists some $c\in I$ with $(a,c)\in R$ and $(c,b)\in\Delta\cup(S\setminus R)$. If $(c,b)\in\Delta$ then $(a,b)=(a,c)\in R\subseteq S$. If $(c,b)\in S\setminus R$ then $(a,c)\in R\subseteq S$ and $(c,b)\in S$ and hence $(a,b)\in S$ according to the transitivity of $S$. This shows $R\circ\big(\Delta\cup(S\setminus R)\big)\subseteq S$. Conversely, assume $(a,b)\in S$. If $(a,b)\in R$ then $(a,b)\in R$ and $(b,b)\in\Delta\cup(S\setminus R)$ and hence $(a,b)\in R\circ\big(\Delta\cup(S\setminus R)\big)$. If $(a,b)\notin R$ then $(a,a)\in R$ according to the reflexivity of $R$ and $(a,b)\in\Delta\cup(S\setminus R)$ and hence $(a,b)\in R\circ\big(\Delta\cup(S\setminus R)\big)$. This shows $S\subseteq R\circ\big(\Delta\cup(S\setminus R)\big)$ completing the proof of the lemma.
\end{proof}

Let $[a_{ij}]$ be an $I\times I$-matrix. Put
\begin{align*}
                    \vec a_k & :=(a_{ik})_{i\in I}\text{ for all }k\in I, \\
       \max_{k\in J}\vec a_k & :=(\max_{k\in J}a_{ik})_{i\in I}\text{ for all }J\subseteq I.
\end{align*}
(We use the convention $\max\limits_{k\in\emptyset}\vec a_k:=(0)_{i\in I}$.)

We can formulate and prove a general result on solving the equation $R\circ X=S$ as follows.

\begin{theorem}\label{th3}
Let $R,S\subseteq I\times I$ and put $M_R=[a_{ij}]$, $M_S=[b_{ij}]$ and $A_i:=\{j\in I\mid a_{ij}=1\}$. Then the following are equivalent:
\begin{enumerate}[{\rm(i)}]
\item The equation $R\circ X=S$ has a solution,
\item For every $k\in I$ there exists some subset $X_k$ of $I$ such that $\max\limits_{l\in X_k}\vec a_l=\vec b_k$ for all $k\in I$. In this case $X=\{(i,j)\mid j\in I,i\in X_j\}$. All solutions can be obtained in this way.
\item For every $k\in I$ there exists some subset $X_k$ of $I$ such that for all $i,k\in I$ we have $A_i\cap X_k=\emptyset$ if and only if $b_{ik}=0$.
\end{enumerate}
\end{theorem}

\begin{proof}
Let $X\subseteq I\times I$ and put $M_X=[x_{ij}]$. \\
(i) $\Leftrightarrow$ (ii): \\
Put $X_k:=\{j\in I\mid x_{jk}=1\}$ for all $k\in I$. Then the following are equivalent:
\begin{align*}
                 R\circ X & =S, \\
             M_R\odot M_X & =M_S, \\
\max_{j\in I}a_{ij}x_{jk} & =b_{ik}\text{ for all }i,k\in I, \\
    \max_{j\in X_k}a_{ij} & =b_{ik}\text{ for all }i,k\in I, \\
  \max_{j\in X_k}\vec a_j & =\vec b_k\text{ for all }k\in I.
\end{align*}
(i) $\Rightarrow$ (iii): \\
Let $X$ be a solution of the equation $R\circ X=S$ and put $X_k:=\{j\in I\mid x_{jk}=1\}$ for all $k\in I$. Then for all $i,k\in I$ the following are equivalent:
\begin{align*}
              A_i\cap X_k & \neq\emptyset, \\
\text{there exists some } & j\in A_i\cap X_k, \\
\text{there exists some } & j\in I\text{ satisfying }a_{ij}=x_{jk}=1, \\
\max_{j\in I}a_{ij}x_{jk} & =1, \\
                   b_{ik} & =1.
\end{align*}
(iii) $\Rightarrow$ (i): \\
Put
\[
x_{ij}:=\left\{
\begin{array}{ll}
1 & \text{if }i\in X_j \\
0 & \text{otherwise}
\end{array}
\right.
\]
for all $i,j\in I$. Then for all $i,k\in I$ we have
\[
\max_{j\in I}a_{ij}x_{jk}=\max_{j\in A_i\cap X_k}1=b_{ik}.
\]
This shows $M_R\odot M_X=M_S$, i.e.\ $R\circ X=S$. 
\end{proof}

Now we will investigate when the incidence matrix $A$ of a binary relation is ``invertible'', it means that there exists an incidence matrix $B$ satisfying $A\odot B=B\odot A=E$ where $E:=[\delta_{ij}]$. For $B$ we will also use the notation $A^{-1}$. (Note that because of the associativity of $\odot$ the inverse, if it exists, is unique.)

\begin{proposition}\label{prop2}
Let $n$ be a positive integer, put $I:=\{1,\ldots,n\}$, let $A=[a_{ij}]\in\{0,1\}^{n\times n}$ and put $E:=[\delta_{ij}]\in\{0,1\}^{n\times n}$. Then the following are equivalent:
\begin{enumerate}[{\rm(i)}]
\item There exists some $B\in\{0,1\}^{n\times n}$ with $A\odot B=B\odot A=E$.
\item There exists some bijection $f\colon I\rightarrow I$ satisfying $a_{ij}=\delta_{j,f(i)}$ for all $i,j\in I$.
\end{enumerate}
\end{proposition}

\begin{proof}
$\text{}$ \\
(i) $\Rightarrow$ (ii): \\
For $j\in I$ let $\vec a_j$ and $\vec e_j$ denote the $j$-th column vector of $A$ and $E$, respectively. Moreover, let $k\in I$. Since $B$ is a solution of the equation $A\odot X=E$, according to Theorem~\ref{th3} there exists some subset $I_k$ of $I$ such that $\max\limits_{l\in I_k}\vec a_l=\vec e_k$. Hence there exists some $f(k)\in I_k$ with $\vec a_{f(k)}=\vec e_k$. Clearly, $f\colon I\rightarrow I$ is injective and thus bijective and we obtain $\vec a_j=\vec e_{f^{-1}(j)}$ for all $j\in I$, i.e.\ $a_{ij}=\delta_{i,f^{-1}(j)}=\delta_{j,f(i)}$ for all $i,j\in I$. \\
(ii) $\Rightarrow$ (i): \\
If $B=[b_{ij}]:=[\delta_{j,f^{-1}(i)}]$ then
\begin{align*}
\max_{k\in I}a_{ik}b_{kj} & =\max_{k\in I}\delta_{k,f(i)}\delta_{j,f^{-1}(k)}=\delta_{f(i),f(j)}=\delta_{ij}, \\
\max_{k\in I}b_{ik}a_{kj} & =\max_{k\in I}\delta_{k,f^{-1}(i)}\delta_{j,f(k)}=\delta_{f^{-1}(i),f^{-1}(j)}=\delta_{ij}
\end{align*}
showing $A\odot B=B\odot A=E$.
\end{proof}

Note that condition (ii) means that every row and every column of $A$ contains exactly one $1$ and that the implication (ii) $\Rightarrow$ (i) also holds for infinite $I$.

\begin{remark}\label{rem1}
Let $R$ and $S$ be binary relations on a set $I$ such that for the incidence matrix $M_R$ of $R$ there exists a bijection $f\colon I\rightarrow I$ as described in Theorem~\ref{th2}. Put $E:=[\delta_{ij}]\in\{0,1\}^{I\times I}$. It is easy to check that $A\odot E=E\odot A=A$ for every $A\in\{0,1\}^{I\times I}$. From Proposition~\ref{prop2} we obtain $M_R^{-1}=[\delta_{j,f^{-1}(i)}]$. Now the following are equivalent:
\begin{align*}
    R\circ X & =S, \\
M_R\odot M_X & =M_S, \\
         M_X & =M_R^{-1}\odot M_S
\end{align*}
and hence $x_{ij}=\max\limits_{k\in I}\delta_{k,f^{-1}(i)}b_{kj}=b_{f^{-1}(i),j}$ for all $i,j\in I$. Note that here we used associativity of $\odot$.
\end{remark}

The next theorem characterizes solvability of the equation $R\circ X=S$ and also characterizes the corresponding solutions. From this theorem we will derive an algorithm for computing all solutions.

\begin{theorem}\label{th4}
Let $R,S\subseteq I\times I$, $M_R=[a_{ij}]$ and $M_S=[b_{ij}]$ and put
\begin{align*}
A_i & :=\{j\in I\mid a_{ij}=1\}, \\
B_k & :=\{i\in I\mid b_{ik}=0\}, \\
C_k & :=\bigcup_{l\in B_k}A_l
\end{align*}
for all $i,k\in I$.
\begin{enumerate}[{\rm(i)}]
\item The equation $R\circ X=S$ is solvable if and only if $A_i\setminus C_k\neq\emptyset$ for all $k\in I$ and all $i\in I\setminus B_k$.
\item $X\subseteq I\times I$ with $M_X=[x_{ij}]$ is a solution of the equation $R\circ X=S$ if and only if the following hold:
\begin{enumerate}[{\rm(a)}]
\item $x_{jk}=0$ for all $k\in I$ and all $j\in C_k$,
\item for every $k\in I$ and $i\in I\setminus B_k$ there exists some $j\in A_i\setminus C_k$ with $x_{jk}=1$.
\end{enumerate}
\end{enumerate}
\end{theorem}

\begin{proof}
Let $X\subseteq I\times I$ and $M_X=[x_{ij}]$. Then the following are equivalent:
\begin{align*}
                 R\circ X & =S, \\
\max_{j\in I}a_{ij}x_{jk} & =b_{ik}\text{ for all }i,k\in I, \\
    \max_{j\in A_i}x_{jk} & =b_{ik}\text{ for all }i,k\in I.
\end{align*}
For $k\in I$ the following are equivalent:
\begin{align*}
\max_{j\in A_i}x_{jk} & =b_{ik}\text{ for all }i\in B_k, \\
\max_{j\in A_i}x_{jk} & =0\text{ for all }i\in B_k, \\
               x_{jk} & =0\text{ for all }i\in B_k\text{ and all }j\in A_i, \\
               x_{jk} & =0\text{ for all }j\in\bigcup_{i\in B_k}A_i, \\
               x_{jk} & =0\text{ for all }j\in C_k.
\end{align*}							
For $k\in I$ and $i\in I\setminus B_k$ the following are equivalent:
\begin{align*}
                               \max_{j\in A_i}x_{jk} & =b_{ik}, \\
                               \max_{j\in A_i}x_{jk} & =1, \\
\text{there exists some }j\in A_i\text{ with }x_{jk} & =1.
\end{align*}							
\end{proof}

As mentioned above we now derive an algorithm for computing all solutions of the equation $R\circ X=S$ provided this equation is solvable and $I$ is finite. This algorithm consists of the following three steps (let $A_i$, $B_k$ and $C_k$ be defined as in Theorem~\ref{th4}):
\begin{enumerate}[(1)]
\item Put $x_{jk}:=0$ for all $k\in I$ and all $j\in C_k$.
\item For all $k\in I$ and $i\in I\setminus B_k$ choose some $j\in A_i\setminus C_k$ and put $x_{jk}:=1$.
\item Choose the remaining $x_{jk}\in\{0,1\}$ arbitrarily.
\end{enumerate}

This algorithm was already implicitly used in Example~\ref{ex1}, see Example~\ref{ex2}. In fact it is similar to the method for solving linear equations. In steps (1) and (2) the algorithm reduces the possibilities for choosing the elements of $M_X$ whereas steps (2) and (3) determine the number of solutions.

The aforementioned algorithm will be demonstrated by the following example.

\begin{example}\label{ex2}
Let us apply the algorithm to Example~\ref{ex1}. Hence we have
\[
M_R=\left(
\begin{array}{ccc}
0 & 1 & 0 \\
1 & 0 & 0 \\
0 & 1 & 1
\end{array}
\right), M_S= \left(
\begin{array}{ccc}
1 & 1 & 0 \\
0 & 0 & 0 \\
1 & 1 & 0
\end{array}
\right).
\]
We compute
\begin{align*}
& A_1=\{2\}, A_2=\{1\}, A_3=\{2,3\}, \\
& B_1=\{2\}, B_2=\{2\}, B_3=\{1,2,3\}, \\
& C_1=\{1\}, C_2=\{1\}, C_3=\{1,2,3\}, \\
& I\setminus B_1=\{1,3\}, I\setminus B_2=\{1,3\}, I\setminus B_3=\emptyset, \\
& A_1\setminus C_1=\{2\}\neq\emptyset, A_3\setminus C_1=\{2,3\}\neq\emptyset, A_1\setminus C_2=\{2\}\neq\emptyset, A_3\setminus C_2=\{2,3\}\neq\emptyset.
\end{align*}
Hence the equation $R\circ X=S$ is solvable and we obtain
\begin{align*}
& x_{11}=x_{12}=x_{13}=x_{23}=x_{33}=0, \\
& 1\in\{x_{21}\}\cap\{x_{21},x_{31}\}\cap\{x_{22}\}\cap\{x_{22},x_{32}\},
\end{align*}
i.e.\ $x_{21}=x_{22}=1$ and $x_{31},x_{32}\in\{0,1\}$. Thus we got all four solutions derived in Example~\ref{ex1}.
\end{example}

There arises the question what can be said concerning the equation $X\circ R=S$.

\begin{remark}
Since the equation $X\circ R=S$ is dual to the equation $R\circ X=S$, the investigation of the first equation does not bring new insights in the problem.
\end{remark}

Authors' addresses:

Ivan Chajda \\
Palack\'y University Olomouc \\
Faculty of Science \\
Department of Algebra and Geometry \\
17.\ listopadu 12 \\
771 46 Olomouc \\
Czech Republic \\
ivan.chajda@upol.cz

Helmut L\"anger \\
TU Wien \\
Faculty of Mathematics and Geoinformation \\
Institute of Discrete Mathematics and Geometry \\
Wiedner Hauptstra\ss e 8-10 \\
1040 Vienna \\
Austria, and \\
Palack\'y University Olomouc \\
Faculty of Science \\
Department of Algebra and Geometry \\
17.\ listopadu 12 \\
771 46 Olomouc \\
Czech Republic \\
helmut.laenger@tuwien.ac.at
\end{document}